\documentclass[12pt]{amsart}
\usepackage{amsmath,amsthm,amssymb,bbm,color,ArticleMacros,verbatim,tikz}

\usetikzlibrary{shapes,backgrounds}

\input xy 
\xyoption{all}

\numberwithin{equation}{section}

\renewcommand{\bar}{\overline}

\newcommand{\cf}{\textrm{cf}}

\newtheorem*{maintheorem}{Main Theorem}
\newtheorem*{keylemma}{Key Lemma}

\begin{document}

\title{The failure of $\GCH$ at a degree of supercompactness}
\author{Brent Cody}
\date{\today}

\begin{abstract}

We determine the large cardinal consistency strength of the existence of a $\lambda$-supercompact cardinal $\kappa$ such that $\GCH$ fails at $\lambda$. Indeed, we show that the existence of a $\lambda$-supercompact cardinal $\kappa$ such that $2^\lambda\geq\theta$ is equiconsistent with the existence of a $\lambda$-supercompact cardinal that is also $\theta$-tall. We also prove some basic facts about the large cardinal notion of tallness with closure.

\end{abstract}

\maketitle
\bigskip


\hyphenation{super-compact}


\section{Introduction}\label{introduction}

Woodin showed that the existence of a measurable cardinal at which $\GCH$ fails is equiconsistent with the existence of a cardinal $\kappa$ that is $\kappa^{++}$-tall (see \cite{Hamkins:TallCardinals}, \cite{Gitik:TheNegationOfTheSingularCardinalHypothesis}, or \cite{Jech:Book}), where a cardinal $\kappa$ is \emph{$\theta$-tall} if there is a nontrivial elementary embedding $j:V\to M$ with critical point $\kappa$ such that $j(\kappa)>\theta$ and $M^\kappa\subseteq M$ in $V$. In this paper we extend Woodin's result into the realm of partially supercompact cardinals. Since $\kappa$ is measurable if and only if $\kappa$ is $\kappa$-supercompact, one immediately sees several natural ways of doing this. Let us consider the following questions for cardinals $\kappa$, $\lambda$, and $\theta$.

\begin{enumerate}
\item What is the strength of the hypothesis that $\kappa$ is $\lambda$-supercompact and $\GCH$ fails at $\kappa$?
\item What is the strength of the hypothesis that $\kappa$ is $\lambda$-supercompact and $\GCH$ fails at $\kappa$ with $2^\kappa\geq\theta$?
\item What is the strength of the hypothesis that $\kappa$ is $\lambda$-supercompact and $\GCH$ fails at $\lambda$?
\item What is the strength of the hypothesis that $\kappa$ is $\lambda$-supercompact and $\GCH$ fails at $\lambda$ with $2^\lambda\geq\theta$?
\end{enumerate}

We note that Woodin's theorem answers question (1) in the case that $\lambda=\kappa$. 

The following theorem, together with Woodin's result, provides complete answers to questions (1) - (4).

\begin{maintheorem} Suppose $\lambda$ and $\theta$ are cardinals. 
\begin{enumerate}
\item For $\lambda>\kappa$, the existence of a $\lambda$-supercompact cardinal $\kappa$ such that $\GCH$ fails at $\kappa$ is equiconsistent with the existence of a $\lambda$-supercompact cardinal.
\item The existence of a $\lambda$-supercompact cardinal $\kappa$ such that $2^\kappa\geq\theta$ is equiconsistent with the existence of a $\lambda$-supercompact cardinal that is also $\theta$-tall.
\item The existence of a $\lambda$-supercompact cardinal $\kappa$ such that $\GCH$ fails at $\lambda$ is equiconsistent with the existence of a $\lambda$-supercompact cardinal that is $\lambda^{++}$-tall.
\item The existence of a $\lambda$-supercompact cardinal $\kappa$ such that the $2^\lambda\geq\theta$ is equiconsistent with the existence of a $\lambda$-supercompact cardinal that is $\theta$-tall.
\end{enumerate}
In each case above, the term ``equiconsistent'' is intended to mean that, in the forward direction the same cardinal witnessing the hypothesis also witnesses the conclusion; and in the reverse direction, the same cardinal witnessing the hypothesis witnesses the conclusion in a forcing extension.

\end{maintheorem}


The details of cardinal preservation in the various forcing extensions in parts (1) - (4) of the main theorem will be worked out below.

Questions (1) - (4) above can be seen as a special case to a more general question: 
\begin{enumerate}
\item[(5)] What kind of $\GCH$ patterns are consistent with a $\lambda$-supercompact cardinal from what type of large cardinal assumption? 
\end{enumerate}
There are some obvious resrtictions such as if $\GCH$ fails at $\kappa$, a $\lambda$-supercom\-pact cardinal, then it must fail unboundedly often below $\kappa$. Also, if $\lambda$ is a strong limit and $\GCH$ holds below and at $\kappa$ then $\GCH$ must hold up to $\lambda$. There are however some subtle issues in answering the general question which we will address in a forthcoming paper.

Let us note that Friedman and Thompson have an alternate approach to proving Woodin's result. Indeed they show in \cite{FriedmanThompson:PerfectTreesAndElementaryEmbeddings} that from the hypothesis of the existence of a $P_2\kappa$-hypermeasurable cardinal $\kappa$, which is equiconsistent with the existence of a cardinal $\kappa$ that is $\kappa^{++}$-tall, after doing a preparatory forcing iteration one may use side-by-side Sacks forcing to pump up the size of the power set of $\kappa$ to $\kappa^{++}$ while preserving the measurability of $\kappa$ using what they refer to as  the tuning fork method.

The way we will establish Math Theorem (2)-(4) in this paper is by forcing that achieves $2^\kappa>\lambda^+$, and hence $2^\lambda>\lambda^+$, and preserves the $\lambda$-supercompactness of $\kappa$, where $\lambda>\kappa$ is a cardinal. This suggests the question, can one force a violation of $\GCH$ at $\lambda$ while preserving $\GCH$ in the interval $[\kappa,\lambda)$ and preserving the $\lambda$-supercompactness of $\kappa$? It seems as though the method of surgical modification of a generic, due to Woodin, does not generalize to answer this question. However, Friedman and Honsik show in their forthcoming paper \cite{FH:Forthcoming} that the answer to this question is yes by using generalized Sacks forcing. We also note that Question (5) above is part of a larger program of attempting to describe what kind of Easton functions can be realized as the continuum function by forcing while preserving large cardinals. For more on this see \cite{FH:Easton}.

Here we provide an outline of the rest of the paper. In section \ref{sectionterminology} we discuss some notational conventions. In section \ref{sectionpreliminaries} we state some basic lemmas about lifting large cardinal embeddings that will be used throughout the paper. We will prove Main Theorem (1) in section \ref{easysection}. In section \ref{preliminariessection}, in order to prepare for the proof of Main Theorem (2) - (4), we discuss the large cardinal concept of ``tallness with closure,'' which synthesizes the concepts of $\lambda$-supercompactness and $\theta$-tallness. We prove Main Theorem (2) - (4) in section \ref{mainproofsection}.

\section{Terminology}\label{sectionterminology}

Let us now give a few remarks about our notational conventions and terminology, which are for the most part standard. For a forcing poset $\P$ and conditions $p$ and $q$ in $\P$ we write $p\leq q$ to indicate that $p$ extends $q$. We say that $\P$ is \emph{$\leq\kappa$-closed} if every descending sequence of conditions of length less than or equal to $\kappa$ has a lower bound. We say that $\P$ is \emph{$<\kappa$-closed} if every descending sequence of conditions of length less than $\kappa$ has a lower bound. We say that $\P$ is \emph{$\leq\kappa$-distributive} if the intersection of $\kappa$ open dense subsets of $\P$ is an open dense subset of $\P$. We will assume that all forcing posets $\P$ have a greatest element which we will denote by $1_\P$. Suppose $M$ and $N$ are models of $\ZFC$ and $\P$ is a forcing notion in $M$. We say that $G\subseteq \P$ is \emph{$M$-generic} for $\P$ if $G$ intersects every dense subset of $\P$ that is in $M$. We will write ``$M^\lambda\subseteq M$ in $N$'' or say that $M$ is closed under $\kappa$-sequences in $N$ to mean that if $\vec{x}$ is a sequence of elements of $M$ of length $\kappa$ in $N$, then $\vec{x}\in M$. If $j:M\to N$ is an elementary embedding between models of $\ZFC$ the \emph{critical point of $j$} is the least ordinal $\alpha$ such that $j(\alpha)\neq\alpha$ which we will denote as $\cp(j)$. For ordinals $\alpha$ we write $\cf(\alpha)$ to denote the cofinality of $\alpha$. If $A$ is a set and $f$ is a function with $A\subseteq \dom(f)$ we write $f"A$ to denote $\{f(a)\mid a\in A\}$.

\section{Basic lemmas on lifting embeddings}\label{sectionpreliminaries}

The large cardinal properties we are interested in, namely partial supercompactness and partial tallness, are witnessed by embeddings $j:M\to N$ and we will be interested in showing that these large cardinal properties are preserved in certain forcing extensions. Here we collect, for conveinence, some terminology and several standard lemmas that will be used to lift embeddings of the form $j:M\to N$ to forcing extensions $j^*:M[G]\to N[j(G)]$. In order to avoid notational complexities we will make use of the standard abuse of notation by using the same symbol, namely $j$, to denote both the ground model embedding $j$ and the lifted embedding $j^*$, even though $j$ and $j^*$ may be classes in different models. After lifting an embedding we will explicitly state in which model the lift is a class. Whenever we say that $j:M\to N$ is an embedding in $M'$ we will mean that $j$ is a nontrivial elementary embedding, $M$, $N$, and $M'$ are models of $\ZFC$, and $j$ is a class of $M'$.
For a detailed discussion of lifting embeddings as well as a proofs of Lemmas \ref{groundlemma} -- \ref{lambdadist}, the reader may consult \cite{Hamkins:Book} or Cummings' article in \cite{HandbookOfSetTheoryVolume2}.

The following two lemmas are useful for building generic objects.

\begin{lemma}\label{groundlemma}
Suppose that $M^\lambda\subseteq M$ in $V$ and there is in $V$ an $M$-generic filter $H\subseteq \Q$ for some forcing $\Q\in M$. Then $M[H]^\lambda\subseteq M[H]$ in $V$.
\end{lemma}

\begin{lemma}\label{chainlemma}
Suppose that $M\subseteq V$ is a model of $\ZFC$, $M^\lambda\subseteq M$ in $V$ and $\P$ is $\lambda^+$-c.c. If $G\subseteq \P$ is $V$-generic, then $M[G]^\lambda\subseteq M[G]$ in $V[G]$.
\end{lemma}

Suppose $j:M\to N$ is an embedding and $\P\in M$ a forcing notion. In order to lift $j$ to $M[G]$ where $G$ is $M$-generic for $\P$, we will often use Lemmas \ref{groundlemma} and \ref{chainlemma} to build an $N$-generic filter $H$ for $j(\P)$ satisfying condition (1) in Lemma \ref{liftingcriterion}.

\begin{lemma}\label{liftingcriterion}
Let $j:M\to N$ be an elementary embedding between transitive models of $\ZFC$. Let $\P\in M$ be a notion of forcing, let $G$ be $M$-generic for $\P$ and let $H$ be $N$-generic for $j(\P)$. Then the following are equivalent. 
\begin{enumerate}
\item $j"G\subseteq H$
\item There exists an elementary embedding $j^*:M[G]\to N[H]$, such that $j^*(G)=H$ and $j^*\restrict M =j$.
\end{enumerate}
\end{lemma}
\noindent We say that the embedding $j^*$ in condition (2) above is a $\emph{lift}$ of $j$.

Suppose $j:V\to M$ is an elementary embedding. We say that a set $S\in V$ \emph{generates $j$ over $V$} if $M$ is of the form 
\begin{align}
M&=\{j(h)(s)\mid h:[A]^{<\omega}\to V, s\in [S]^{<\omega}, h\in V\}.\label{seedrep}
\end{align}
where $A\in V$ and $S\subseteq j(A)$. We will often make use of the following lemma which states that the above representation (\ref{seedrep}) of the target model of an elementary embedding remains valid after forcing.

\begin{lemma}\label{seedpreservationlemma}
If $j:V\to M$ is an elementary embedding generated over $V$ by a set $S\in V$ then any lift of this embedding to a forcing extension $j^*:V[G]\to M[j^*(G)]$ is generated by $S$ over $V[G]$.
\end{lemma}

We will often make use of the next standard lemma which states that embeddings witnessed by extenders are preserved by highly distributive forcing.

\begin{lemma}\label{lambdadist}
Suppose $j:V\to M$ is an elementary embedding with critical point $\kappa$ and $M=\{j(h)(s)\mid h:A\to V, s\in [S]^{<\omega}, \textrm{ and } h\in V\}$ where $|A|=\lambda$ and $S\subseteq j(A)$ is some set in $V$. If $\P\in V$ is $\leq\lambda$-distributive forcing then $j"G$ generates an $M$-generic filter for $j(\P)$ and $j$ lifts uniquely to $j^*:V[G]\to M[H]$ in $V[G]$ where $j^*(G)=H$ is the filter on $j(\P)$ generated by $j"G$.
\end{lemma}

\section{Proof of Main Theorem (1)}\label{easysection}

Our proof of the main theorem will use a preparatory forcing notion called the lottery preparation, which was introduced by Hamkins in \cite{Hamkins:TheLotteryPreparation}. The lottery preparation works uniformly as a generalized Laver preparation in a variety of large cardinal contexts. Here we give a brief introduction to the lottery preparation.

The \emph{lottery sum} of a collection of posets $\{(\Q_\alpha,\leq_\alpha)\mid\alpha<\kappa\}$ is
$$\bigoplus\{\Q_\alpha\mid\alpha<\kappa\}:=\{\emptyset \}\cup\bigcup_{\alpha<\kappa}\{(\alpha,q)\mid q\in \Q_\alpha\}$$
where the ordering on the lottery sum is defined by (1) $(\alpha,q)\leq\emptyset$ for all $\alpha<\kappa$ and $q\in \Q_\alpha$ and (2) $(\alpha,p)\leq(\beta,q)$ if and only if $\alpha=\beta$ and $p\leq_\alpha q$. As Hamkins says, a generic for the lottery sum of a collection of posets chooses a poset and forces with it. For a detailed account of the lottery preparation see \cite{Hamkins:TheLotteryPreparation}.

We now define the lottery preparation of $\kappa$, which is a reverse Easton support iteration of length $\kappa$. We say a poset $\Q$ is \emph{allowed} at stage $\gamma$ if $\Q$ is $<\gamma$-strategically closed; note that ``$<\gamma$-strategic closure'' will not play a role in this paper so the reader that is unfamiliar with this concept may take this to simply mean $<\gamma$-closed. For a partial function $f\subseteq\kappa\times\kappa$ we define the \emph{lottery preparation of $\kappa$ with respect to $f$} to be the reverse Easton support forcing iteration such that if $\gamma<\kappa$ is inaccessible and $f"\gamma\subseteq\gamma$ then the stage $\gamma$ forcing is the lottery sum in $V^{\P_\gamma}$ of all allowed posets in $H(f(\gamma)^+)$ and otherwise the stage $\gamma$ forcing is trivial.  Suppose $\P$ is the lottery preparation of $\kappa$ with respect to a partial function $f\subseteq\kappa\times\kappa$ and that $\Q$ is a poset that appears in the stage $\gamma<\kappa$ lottery sum. Then the condition $\langle p_\alpha\mid\alpha<\kappa\rangle$ with $p_\alpha=\emptyset$ for $\alpha\neq\gamma$ and $p_\gamma=1_\Q$ \emph{opts} for $\Q$ at stage $\gamma$ where $1_\Q$ is a $\P_\gamma$-name for the top element of $\Q$.


The lottery preparation $\P$ of some large cardinal $\kappa$ is usually used with respect to a partial function $f\subseteq\kappa\times\kappa$ with the Menas property, such as a function added by fast function forcing (see section \ref{fffsection}). Using the lottery preparation with respect to such a function insures that $j(\P)$, where $j$ is an elementary embedding witnessing the large cardinal property at hand, has a tail with a high degree of closure.

We will now show that given a $\lambda$-supercompact cardinal $\kappa$, one may pump up the power set of $\kappa$ to have size at least $\lambda^+$ while maintaining the $\lambda$-supercompactness of $\kappa$. This will establish Main Theorem (1) because if we assume $\lambda>\kappa$ then in a forcing extension we will have $2^\kappa\geq\lambda^+\geq\kappa^{++}$. We note that this result essentially follows from the methods of \cite{ApterHamkins:IndestructibilityAndTheLevelByLevelAgreement}.

\begin{theorem}\label{easy}
If $\kappa$ is $\lambda$-supercompact then there is a forcing extension preserving this in which $2^\kappa\geq\lambda^+$.
\end{theorem}

\begin{proof} Since any $\lambda$-supercompact cardinal is also $\lambda^{<\kappa}$-super\-compact we may assume without loss of generality that $\lambda^{<\kappa}=\lambda$. We assume $2^\kappa\leq\lambda$ because otherwise the theorem is trivial. We may further assume that $2^{\lambda}=\lambda^+$ since the forcing to achieve this is $\leq\lambda$-distributive and hence preserves the $\lambda$-supercompactness of $\kappa$. Let $j:V\to M$ witness that $\kappa$ is $\lambda$-supercompact. By Lemma \ref{seedpreservationlemma}, we may assume that each element of $M$ is of the form $j(h)(j"\lambda)$ where $h:P_\kappa\lambda\to V$ is in $V$. Note that $(\lambda^+)^M=\lambda^+$ since $M$ is closed under $\lambda$ sequences in $V$. Thus since $j(\kappa)$ is inaccessible in $M$ we have that $j(\kappa)>\lambda^+$. Now let $\P$ be the lottery preparation of $\kappa$ relative to a partial function $f\subseteq\kappa\times\kappa$ with the Menas property $j(f)(\kappa)>\lambda$. Let $G$ be $V$-generic for $\P$. Let $\Q:=\Add(\kappa,\lambda^+)^{V[G]}$ and let $H$ be $V[G]$-generic for $\Q$. 

By elementarity, $j(\P)$ is the lottery preparation of $j(\kappa)$ defined relative to $j(f)$. Since $M$ is closed under $\lambda$-sequences in $V$ we know that the first $\kappa$ stages of $\P$ and $j(\P)$ agree. Since $j(f)\restrict\kappa=f$ it follows that $j(f)"\kappa\subseteq\kappa$ and since $\kappa$ is inaccessible in $M$ we see that the stage $\kappa$ forcing in $j(\P)$ is the lottery sum in $M[G]$ of all allowed posets in $H(j(f)(\kappa)^+)$. Since $\Q$ is in $M[G]$ and also in $H(j(f)(\kappa)^+)^{M[G]}$ it follows that $\Q$ appears in the stage $\kappa$ lottery sum in $j(\P)$. Thus, we may factor $j(\P)$ below a condition $p$ that opts for $\Q$ at stage $\kappa$ as $j(\P)\restrict p\cong \P * \Q * \P_{tail}$ where $\P_{tail}$ is a term for the forcing $j(\P)$ beyond stage $\kappa$. For example, $p$ could be the condition $\langle p_\alpha \mid \alpha<j(\kappa)\rangle$ such that $p_\kappa=1_\Q$ and $p_\alpha=\emptyset$ for every other $\alpha<j(\kappa)$. Since $j(f)(\kappa)>\lambda$ and nontrivial forcing occurs in $j(\P)$ only at closure points of $j(f)$, we see that the next stage of nontrivial forcing beyond $\kappa$ in $j(\P)$ is indeed beyond $\lambda$ and from this it follows that $\P_{tail}$ is a term for $\leq\lambda$-closed forcing. Since $M\subseteq V$ we see that $G$ is $M$-generic for $\P$ and $H$ is $M[G]$-generic for $\Q$. Thus $\P_{tail}$ is $\leq\lambda$-closed in $M[G][H]$. Furthermore, it follows from Lemma \ref{chainlemma} that $M[G][H]$ is closed under $\lambda$-sequences in $V[G][H]$ because $\P*\Q$ is $\kappa^+$-c.c.. Since in $V$, $\P$ has at most $2^\kappa\leq\lambda$-many dense subsets we see that $\P_{tail}$ has at most $j(\lambda)$-many dense subsets in $M[G][H]$ where $|j(\lambda)|^V\leq (\lambda)^{\lambda^{<\kappa}}=\lambda^\lambda=2^\lambda=\lambda^+$. Thus in $V[G][H]$, by constructing a descending sequence of conditions, we may build an $M[G][H]$-generic for $\P_{tail}$, call it $G_{tail}$. We may now lift the embedding to $j:V[G]\to M[j(G)]$ where $j(G)=G*H*G_{tail}$ and the lifted embedding is a class of $V[G][H]$. It follows from Lemma \ref{groundlemma} that $M[j(G)]$ is closed under $\lambda$ sequences in $V[G][H]$.

Now we lift the embedding through $\Q$. To do this we follow the method used in \cite{ApterHamkins:IndestructibilityAndTheLevelByLevelAgreement} Corollary 10. Let $A\subseteq j(\Q)=\Add(j(\kappa),j(\lambda^+))$ be a maximal antichain in $M[j(G)]$. Let $r\in j(\Q)$ be a condition that is compatible with every element of $j"H$. We will show that there is a condition $r'\leq r$ that decides $A$ that is still compatible with every element of $j"H$. Since $j(\Q)$ is $j(\kappa^+)$-c.c. we know that $|A|\leq j(\kappa)$ in $M[j(G)]$. Since $\cf(\lambda^+)>\lambda$ we have $\sup j"\lambda^+=j(\lambda^+)$ and this implies that $A\subseteq \Add(j(\kappa),j(\alpha))$ for some $\alpha<\lambda^+$. We fix such an $\alpha$ so that also $r\in \Add(j(\kappa),j(\alpha))$. Let $q=\bigcup(j"(H\cap\Add(\kappa,\alpha)))$. Since $j(p)=j"p$ for $p\in \Add(\kappa,\alpha)$ we have $|q|\leq\lambda<j(\kappa)$ and thus $q\in M[j(G)]$ is a master condition in $\Add(j(\kappa),j(\alpha))$ (which is a complete subposet of $\Add(j(\kappa),j(\lambda^+))$). Now since $r$ is compatible with every element of $j"H$ we see that $r$ and $q$ are compatible in $\Add(j(\kappa),j(\alpha))$.  Choose $r'\in \Add(j(\kappa),j(\alpha))$ below $r$ and $q$ deciding $A$. We will show that $r'$ remains compatible with $j"H$. Consider $j(p)$ for $p\in H$. We may split $p$ into two pieces: $p=p_0\cup p_1$ where $\dom(p_0)\subseteq \alpha\times\kappa$ and $\dom(p_1)\subseteq [\alpha,\lambda^+)\times\kappa$. Then $j(p)=j(p_0)\cup j(p_1)$ where the domain of $j(p_1)$ is disjoint from the domain of any element of $\Add(j(\kappa),j(\alpha))$. Thus, $r'$ is compatible with $j(p_1)$ in $\Add(j(\kappa),j(\lambda^+))$. Furthermore, we have $r'\leq q\leq j(p_0)$ and hence $r'$ is compatible with $j(p)$.


Since $\Q$ has $\lambda^+$-many antichains we may iterate this to choose a decreasing sequence of conditions in $V[G][H]$ meeting all the antichains of $\Add(j(\kappa),j(\lambda^+))$ such that each element of the sequence is compatible with $j"H$. Let $j(H)$ be the filter generated by this sequence. Then $j(H)$ is an $M[j(G)]$-generic for $\Add(j(\kappa),j(\lambda^+))$ with $j"H\subseteq j(H)$. Hence we may lift the embedding to $j:V[G][H]\to M[j(G)][j(H)]$ in $V[G][H]$, which implies that $\kappa$ is $\lambda$-supercompact in $V[G][H]$.

\end{proof}

\section{Tallness with closure}\label{preliminariessection}

\subsection{Definitions and basic facts}\label{tallness}

Here we include some basic definitions and results about $\theta$-tall cardinals, and $\theta$-tall cardinals with closure $\lambda$, where $\lambda$ is some cardinal and $\theta$ is an ordinal. A cardinal $\kappa$ is called \emph{$\theta$-tall} if there is a nontrivial elementary embedding $j:V\to M$ with critical point $\kappa$ such that $j(\kappa)>\theta$ and $M^\kappa\subseteq M$. Woodin and Gitik used such cardinals to determine the strength of the failure of $\GCH$ at a measurable cardinal (see \cite{Gitik:TheNegationOfTheSingularCardinalHypothesis}), and Hamkins has studied them in their own right in \cite{Hamkins:TallCardinals}. Hamkins says that $\kappa$ is \emph{$\theta$-tall with closure $\lambda$} if there is an elementary embedding $j:V\to M$ with $\cp(j)=\kappa$, $j(\kappa)>\theta$, and $M^\lambda\subseteq M$ in $V$. By following a $\lambda$-supercompactness embedding with a $\theta$-tallness embedding one may show that if $\kappa$ is $\theta$-tall and $\lambda$-supercompact, then $\kappa$ is $\theta$-tall with closure $\lambda$. Indeed, a cardinal $\kappa$ is $\theta$-tall and $\lambda$-supercopmact if and only if it is $\theta$-tall with closure $\lambda$.

We will need the following lemma.

\begin{lemma}\label{generators}
If $\kappa$ is $\theta$-tall with closure $\lambda$ then there is an embedding witnessing this $j:V\to M$ such that $$M=\{j(h)(j"\lambda,\alpha)\mid \alpha\leq\delta \textrm{ and } h:P_\kappa\lambda\times\kappa\to V \textrm{ is a function in $V$}\}$$ where $\delta=(\theta^\lambda)^M$. 
\end{lemma}

\begin{proof}
Let $j_0:V\to M_0$ witness the $\theta$-tallness with closure $\lambda$ of $\kappa$ and let $$X=\{j_0(h)(j_0"\lambda,\alpha)\mid \alpha\leq\delta \textrm{ and } h:P_\kappa\lambda\times\kappa\to V \textrm{ with } h\in V\}$$ where $\delta:=(\theta^\lambda)^M$. It is routine to verify that $X\elesub M_0$. Let $\pi:X\to M$ be the Mostowski collapse of $X$ and define an elementary embedding $j:V\to M$ by $j= \pi \circ j_0$ an let $k:=\pi^{-1}:M\to X\subseteq M_0$. It follows that $j$ is the desired embedding.
\end{proof}

We will often make use of the easy fact that if $\kappa$ is $\theta$-tall with closure $\lambda$, then it is $\theta^\lambda$-tall with closure $\lambda^{<\kappa}$, which we demonstrate now. If $\kappa$ is $\lambda$-supercopmact it is easy to see that $\kappa$ must be $\lambda^{<\kappa}$-supercompact. By following a $\lambda^{<\kappa}$-supercompactness embedding by a $\theta$-tallness embedding we obtain $j:V\to M$ witnessing that $\kappa$ is $\theta$-tall with closure $\lambda^{<\kappa}$, then since $j(\kappa)$ is inaccessible in $M$ and $M^{\lambda^{<\kappa}}\subseteq M$, we have $\theta^\lambda \leq (\theta^\lambda)^M <j(\kappa)$. Thus $j$ witnesses that $\kappa$ is $\theta^\lambda$-tall with closure $\lambda^{<\kappa}$.

By the remarks in the previous paragraph, given that $\kappa$ is $\theta$-tall with closure $\lambda$, in many arguments we will be able to assume without loss of generality that $\theta^\lambda=\theta$ and $\lambda^{<\kappa}=\lambda$. Then by Lemma \ref{generators} there is an embedding $j:V\to M$ witnessing that $\kappa$ is $\theta$-tall with closure $\lambda$ such that $$M=\{j(h)(j"\lambda,\alpha)\mid\textrm{$\alpha\leq\theta$ and $h:P_\kappa\lambda\times\kappa\to V$ is a function in $V$}\}.$$

\subsection{Fast function forcing and tallness with closure}\label{fffsection}

The goal of this section will be to prove that we can force to add a function with the Menas property with respect to $\theta$-tallness with closure $\lambda$. In other words, we will show that if $j:V\to M$ witnesses that $\kappa$ is $\theta$-tall with closure $\lambda$, then we may force to add a partial function $f\subseteq \kappa\times\kappa$ with the property $j(f)(\kappa)>\theta$. In fact, we can arrange any particular value for $j(f)(\kappa)$ up to $j(\kappa)$, the degree of tallness of $\kappa$. To accomplish this we will use a technique invented by Woodin called fast function forcing.

For a cardinal $\kappa$ we define the \emph{fast function forcing} poset $\F_\kappa$ as follows. Conditions in $\F_\kappa$ are partial functions $p\subseteq\kappa\times\kappa$ such that
\begin{enumerate}
\item each $\gamma\in \dom(p)$ is inaccessible and $p"\gamma\subseteq\gamma$,
\item for each inaccessible $\gamma\leq\kappa$ we have $|p\restrict\gamma|<\gamma$.
\end{enumerate}
The ordering on $\F_\kappa$ is given by $p\leq q$ if and only if $p\supseteq q$. For a fixed condition of the form $p:=\{(\gamma,\delta)\}$ we may factor $\F_\kappa$ below $p$ as $\F_\kappa\restrict p \cong \F_{\gamma}\times\F_{[\lambda,\kappa)}$ where $\lambda$ is the next inaccessible beyond $\max(\gamma,\delta)$ and $\F_{[\lambda,\kappa)}:=\{p\in \F_\kappa\mid \dom(p)\subseteq [\lambda,\kappa)\}$. A generic $G$ for $\F_\kappa$ provides a partial function $f:=\bigcup G$ from $\kappa$ to $\kappa$. Since we will only be concerned with the function $f$, and $f$ determines $G$, we will write the forcing extension by the fast-function-forcing poset as $V[f]$ from this point forward. For a more detailed account of fast-function-forcing see \cite{Hamkins:TheLotteryPreparation}.

\begin{lemma}\label{menaslemma}
Suppose $j:V\to M$ is a $\theta$-tallness embedding with closure $\lambda$ with critical point $\kappa$ where $\lambda\leq\theta$ (or merely $\lambda$ is less than the first inaccessible beyond $\theta$). Then there is a fast function forcing extension $V[f]$ such that $j$ lifts to $j:V[f]\to M[j(f)]$ witnessing the $\theta$-tallness with closure $\lambda$ in $V[f]$ such that $j(f)(\kappa)>\theta$. Furthermore, for any $\delta<j(\kappa)$ there is such a lift $j$ such that $j(f)(\kappa)=\delta$.
\end{lemma}

\begin{proof}
As mentioned at the end of subsection \ref{tallness}, we may assume without loss of generality that $\lambda^{<\kappa}=\lambda$ and $\theta^\lambda=\theta$. We may assume that $2^\lambda=\lambda^+$ since this can be accomplished using $\leq\lambda$-distributive forcing, which preserves the $\theta$-tallness with closure $\lambda$ of $\kappa$ by Lemma \ref{lambdadist}. Let $f$ be $V$-generic for $\F_\kappa$ and let $j:V\to M$ be a $\theta$-tallness embedding with closure $\lambda$ such that $$M=\{j(h)(j"\lambda,\alpha)\mid\textrm{$\alpha\leq\theta$ and $h:P_\kappa\lambda\times\kappa\to V$ is a function in $V$}\}.$$ Let $\delta$ be an ordinal with $\theta<\delta<j(\kappa)$ and let $p:=\{(\kappa,\delta)\}$. We may factor $j(\F_\kappa)\restrict p\cong \F_\kappa\times \F_{[\gamma,j(\kappa))}$ where $\gamma$ is the next inaccessible cardinal above $\delta$.

We would like to build an $M[f]$-generic filter for $\F_{[\gamma,j(\kappa))}$ in $V[f]$. Let $D$ be a dense subset of $\F_{[\gamma,j(\kappa))}$ in $M[f]$. Then $D$ has an $\F_\kappa$ name $\dot{D}\in M$ and $\dot{D}=j(h_{\dot{D}})(j"\lambda,\alpha)$ for some $\alpha\leq\theta$ and $h_{\dot{D}}:P_\kappa\lambda\times\kappa\to V$. Since $j"\lambda,j(h_{\dot{D}})\in M$ it follows that $\vec{D}:=\langle j(h_{\dot{D}})(j"\lambda,\alpha)\mid\alpha\leq\theta\rangle\in M$. Using the $\leq\theta$-closure of $\F_{[\gamma,j(\kappa))}$ in $M[f]$ we may find a single condition in $\F_{[\gamma,j(\kappa))}$ meeting every dense set mentioned by $\vec{D}_f:=\langle j(h_{\dot{D}})(j"\lambda,\alpha)_f\mid\alpha\leq\theta\rangle$.

Now we can assume without loss of generality that $h_{\dot{D}}:P_\kappa\lambda\times\kappa\to\{\textrm{nice names of dense subsets of a tail of $\F_\kappa$}\}$. Since $|\F_\kappa|=\kappa$ it follows that there are $2^\kappa$-many nice names for dense subsets of a tail of $\F_\kappa$. This implies that there are $(2^\kappa)^{\lambda^{<\kappa}}=2^\lambda=\lambda^+$-many functions $h$ with domain $P_\kappa\lambda\times\kappa$ that represent nice names for dense subsets of a tail of $\F_\kappa$. In $V$ we may enumerate such $h$'s as $\langle h_\xi\mid\xi<\lambda^+\rangle$. Since every dense subset of $\F_{[\gamma,j(\kappa))}$ in $M[f]$ has a nice name and each nice name is represented by a function $h_\xi$ on our list, we may build an $M[f]$-generic for $\F_{[\kappa,j(\kappa))}$ in $V[f]$ as follows. At a successor stage $\xi$, by using the $\leq\theta$-closure of $\F_{[\gamma,j(\kappa))}$ in $M[f]$ we find a single condition $p_\xi\in \F_{[\gamma, j(\kappa))}$ below all previously constructed conditions meeting each dense set of the form $j(h_\xi)(j"\lambda,\alpha)_f$ for $\alpha\leq\theta$. At limit stages we use the fact that $\F_{[\kappa,j(\kappa))}$ is $<\lambda^+$-closed in $V[f]$ to find a condition below all previously constructed conditions. This defines a descending sequence of conditions in $V[f]$ and we let $f_{[\gamma,j(\kappa))}$ be the $M[f]$-generic filter for $\F_{[\gamma,j(\kappa))}$ generated by the sequence. Since $j"f\subseteq f\cup p\cup f_{[\gamma,j(\kappa))}$ we may lift $j$ to $j:V[f]\to M[j(f)]$ where $j(f)=f\cup p \cup f_{[\gamma,j(\kappa))}$ and $j$ is a class of $V[f]$. Since $F_\kappa$ is $\kappa^+$-c.c. and $f_{[\gamma,j(\kappa))}$ is in $V[f]$ it follows by Lemmas \ref{groundlemma} and \ref{chainlemma} that $M[j(f)]$ is closed under $\lambda$-sequences in $V[f]$ and hence that the lifted embedding witnesses that $\kappa$ is $\theta$-tall with closure $\lambda$ in $V[f]$.

\end{proof}

\subsection{The lottery preparation and tallness with closure}

In \cite{Hamkins:TheLotteryPreparation}, Hamkins shows that the lottery preparation makes many large cardinals indestructible by a wide array of forcing notions. Here we will extend the results in \cite{Hamkins:TheLotteryPreparation} to include $\theta$-tallness with closure $\lambda$.


\begin{theorem}\label{indestructibility}
Suppose $\kappa$ is $\theta$-tall with closure $\lambda$ where $\lambda\leq\theta$. Then after the lottery preparation, the $\theta$-tallness with closure $\lambda$ is indestructible by $<\kappa$-directed closed forcing of size $\leq\lambda$ and $\lambda$ is preserved.
\end{theorem}

\begin{proof}
Suppose $j:V\to M$ witnesses the $\theta$-tallness with closure $\lambda$ of $\kappa$. Without loss of generality we may assume that $\lambda^{<\kappa}=\lambda$, $\theta^\lambda=\theta$, $2^\lambda=\lambda^+$, and $$M=\{j(h)(j"\lambda,\alpha)\mid\textrm{$\alpha\leq\theta$ and $h:P_\kappa\lambda\times\kappa\to V$ is a function in $V$}\}.$$
We remark that the forcing to obtain $2^\lambda=\lambda^+$ collapses cardinals in the interval $[\lambda^+,2^\lambda]$ to $\lambda^+$. By Lemma \ref{menaslemma} we may assume there is a fast function $f\subseteq\kappa\times\kappa$ with $j(f)(\kappa)>\theta$. Let $\P$ be the lottery preparation defined relative to $f$ and let $G$ be $V$-generic for $\P$. Let $\Q$ be any $<\kappa$-directed closed forcing of size $\leq\lambda$ in $V[G]$ and let $H$ be $V[G]$-generic for $\Q$.

Since $\Q$ could be trivial forcing it will suffice to lift $j$ to $V[G][H]$ in $V[G][H]$. We assume without loss of generality that $\Q\subseteq \ORD$. Since $|\P|^V=\kappa$ it follows that $M[G]^\lambda\subseteq M[G]$ in $V[G]$ and hence $\Q\in M[G]$. By elementarity, $j(\P)$ is the lottery preparation of $j(\kappa)$ with respect to $j(f)$. Since $M$ is closed under $\lambda$-sequences in $V$ it follows that the first $\kappa$ stages in $\P$ and $j(\P)$ are the same. Since $\lambda\leq\theta$ we have $\Q\in H(j(f)(\kappa)^+)$ and thus $\Q$ appears in the lottery sum at stage $\kappa$ in $j(\P)$. Thus we may factor $j(\P)$ below a condition $p$ that opts for $\Q$ at stage $\kappa$ as $j(\P)\restrict p\cong \P*\Q*\P_{tail}$, where $\P_{tail}$ is a term for the iteration beyond stage $\kappa$. We know that $\P_{tail}$ is a term for $\leq\theta$-closed forcing because $j(f)(\kappa)>\theta$. Since $|\Q|^{V[G]}\leq\lambda$, it follows that $M[G][H]^\lambda\subseteq M[G][H]$ in $V[G][H]$. We will now construct an $M[G][H]$-generic for $\P_{tail}$ in $V[G][H]$. Let $D$ be a dense subset of $\P_{tail}$ in $M[G][H]$. Let $\dot{D}\in M$ be a $\P*\Q$-name for $D$, that is $\dot{D}_{G*H}=D$, and let $\dot{D}=j(h_{\dot{D}})(j"\lambda,\alpha)$ where $h_{\dot{D}}:P_\kappa\lambda\times\kappa\to V$, $h_{\dot{D}}\in V$, and $\alpha\leq\theta$. Since $j(h_{\dot{D}}),j"\lambda\in M$ the sequence of names $\vec{D}:=\langle j(h_{\dot{D}})(j"\lambda,\alpha)\mid\alpha\leq\theta\rangle$ is in $M$, and furthermore the sequence of dense subsets of $\P_{tail}$, $\vec{D}_{G*H}:=\langle j(h_{\dot{D}})(j"\lambda,\alpha)_{G*H}\mid\alpha\leq\theta\rangle$, is in $M[G][H]$. Since $\P_{tail}$ is $\leq\theta$-closed in $M[G][H]$ we can find a single condition below every dense set mentioned by $\vec{D}_{G*H}$. Without loss of generality we may assume that the range of $h_{\dot{D}}$ is contained in the set of nice names for dense subsets of a tail of $\P$. Now working in $V[G][H]$ we put a bound on the number of functions $$h:P_\kappa\lambda\times\kappa\to \{\textrm{nice names for dense subsets of a tail of $\P$}\}.$$ Since $|\P|=\kappa$ there are $2^\kappa$-many nice names for subsets of a tail of $\P$. Thus there are at most $(2^\kappa)^{\lambda^{<\kappa}}=2^\lambda=\lambda^+$-many such $h$'s. 
In $V[G][H]$ we may enumerate all such $h$'s as $\langle h_\xi\mid\xi<\lambda^+\rangle.$ Since every dense subset of $\P_{tail}$ in $M[G][H]$ has a nice $\P*\Q$-name and each nice $\P*\Q$-name is represented by one of the $h_\xi$'s on our list, we may construct an $M[G][H]$-generic descending sequence of conditions of $\P_{tail}$ as follows. At successor stages $\xi$, we work in $M[G][H]$ and use the fact that $\P_{tail}$ is $\leq\theta$-closed in $M[G][H]$ to find a condition of $\P_{tail}$ meeting every dense subset of $\P_{tail}$ which has a name on the list $\langle j(h_\xi)(j"\lambda,\alpha)\mid\alpha\leq\theta\rangle$. At limits $\xi<\lambda^+$, since $M[G][H]$ is closed under $\lambda$-sequences in $V[G][H]$ it follows that $\P_{tail}$ is $\leq\lambda$-closed in $V[G][H]$, and hence, in $V[G][H]$, we may find a condition of $\P_{tail}$ below all previously constructed conditions. This defines a descending $\lambda^+$-sequence of conditions in $\P_{tail}$ and we let $G_{tail}$ be the filter generated by this sequence of conditions. Clearly $G_{tail}\in V[G][H]$ is an $M[G][H]$-generic filter for $\P_{tail}$. Thus we may lift the embedding to $j:V[G]\to M[j(G)]$ where $j(G):=G*H*G_{tail}$ and $j$ is a class of $V[G][H]$. Since $\P*\Q$ is $\lambda^+$-c.c. it follows from Lemma \ref{chainlemma} that $M[G][H]^\lambda\subseteq M[G][H]$ in $V[G][H]$. Furthermore, since $G_{tail}\in V[G][H]$, we see by Lemma \ref{groundlemma} that $M[j(G)]$ is closed under $\lambda$-seqences in $V[G][H]$.

We will now lift $j$ to $V[G][H]$. Since $H,j"\Q\in M[j(G)]$ we may build $j"H$ in $M[j(G)]$. Since $j(\Q)$ is $<j(\kappa)$-directed closed in $M[j(G)]$ it follows that in $M[j(G)]$, there is a master condition $r\in j(\Q)$ below each element of $j"H$. We will now construct an $M[j(G)]$-generic filter for $j(\Q)$ in $V[G][H]$. Let $D\in M[j(G)]$ be a dense subset of $j(\Q)$. Then by Lemma \ref{seedpreservationlemma} we may write $D=j(h_D)(j"\lambda,\alpha)$ where $h_D\in V[G]$ is a function from $P_\kappa\lambda\times\kappa$ to the collection of dense subsets of $\Q$ and $\alpha\leq\theta$. Now let $\vec{D}:=\langle j(h_D)(j"\lambda,\alpha)\mid\alpha\leq\theta\rangle$. Since $j"\lambda,j(h_D)\in M[j(G)]$ we see that $\vec{D}\in M[j(G)]$. Since $j(\Q)$ is $<j(\kappa)$-directed closed in $M[j(G)]$ we can find, via an internal argument in $M[j(G)]$, a single condition that meets every dense set mentioned by $\vec{D}$. In $V[G]$, $|\Q|=\lambda$ and this implies that there are at most $(2^\lambda)^{\lambda^{<\kappa}}=\lambda^+$-many functions $h\in V[G]$ that represent dense subsets of $j(\Q)$ in $M[j(G)]$. As before, we enumerate these functions as $\langle h_\xi\mid\xi<\lambda^+\rangle$ and define a descending sequence of conditions meeting every dense subset of $j(\Q)$. We start the descending sequence with the master condition, $r$. If $\xi$ is a successor, we use the $<j(\kappa)$-directed closure of $j(\Q)$ in $M[j(G)]$ to meet all dense sets mentioned in $\langle j(h_\xi)(j"\lambda,\alpha)\mid\alpha\leq\theta\rangle$ with a single condition that is also below $r$. At limit stages $\xi<\lambda^+$ since $M[j(G)]$ is closed under $\lambda$-sequences in $V[G][H]$ it follows that $j(\Q)$ is $\leq\lambda$-closed in $V[G][H]$, and hence we may find a condition of $j(\Q)$ below all previously constructed conditions. This defines a descending $\lambda^+$-sequence of conditions below the master condition $r$. Let $j(H)$ be the generic filter generated by this sequence. Since $r$ is stronger than every element of $j"H$ and $r\in j(H)$ we have $j"H\subseteq j(H)$ and thus we may lift the embedding to $j:V[G][H]\to M[j(G)][j(H)]$ where the lifted embedding is a class in $V[G][H]$.

This shows that the $\theta$-tallness with closure $\lambda$ of $\kappa$ is indestructible by any $<\kappa$-directed closed forcing of size $\leq\lambda$ in $V[G]$.

\end{proof}

Let us now give a quick application of Theorem \ref{indestructibility}.

\begin{corollary}\label{GCHoninterval}
If $\kappa$ is $\theta$-tall with closure $\lambda$ where $\lambda\leq\theta$ and $\lambda$ is inaccessible then there is a forcing extension preserving the inaccessibility of $\lambda$ in which $\kappa$ is $\theta$-tall with closure $\lambda$ and $\GCH$ holds on $[\kappa,\lambda]$.
\end{corollary}

\begin{proof}
By Theorem \ref{indestructibility} we may assume that the $\theta$-tallness with closure $\lambda$ is indestructible by $<\kappa$-directed closed forcing of size $\leq\lambda$. We define a length $\lambda$ forcing iteration $\P_\lambda$ with reverse Easton support as follows. Let the first $\kappa$ stages of $\P_{\lambda}$ be trivial forcing. For $\kappa\leq\gamma<\lambda$ we force at cardinal stages $\gamma$ with $\Add(\gamma^+,1)^{V_{\P_\gamma}}$. Clearly $\P_\lambda$ is $\leq\kappa$-directed closed and thus $\leq\kappa$-distributive. Let $G$ be $V$-generic for $\P$. 

It is routine to show that $\lambda$ remains inaccessible in $V[G]$ and is thus not collapsed by the forcing $G$.


Since $\lambda$ is inaccessible it follows that $\P$ has size at most $2^{<\lambda}\leq\lambda$. Furthermore, $\P$ is $<\kappa$-directed closed.
Hence by Theorem \ref{indestructibility} it follows that in $V[G]$, $\kappa$ is $\theta$-tall with closure $\lambda$ and $\GCH$ holds on $[\kappa,\lambda)$. Now we may force $\GCH$ to hold at $\lambda$ with $\leq\lambda$-distributive forcing $\Q=\Add(\lambda^+,1)$, which clearly preserves the $\theta$-tallness with closure $\lambda$ of $\kappa$ by Lemma \ref{lambdadist}.

\end{proof}

\section{Proof of Main Theorem (2) - (4)}\label{mainproofsection}

Let us now argue that the equiconsistencies in the forward directions in Main Theorem (2) - (4) are actually implications. For Main Theorem (2), suppose $j:V\to M$ witnesses that $\kappa$ is $\lambda$-supercompact and $2^\kappa\geq\theta$. Since $j(\kappa)$ is inaccessible in $M$ we have $\theta\leq2^\kappa\leq (2^\kappa)^M<j(\kappa)$. Hence $j$ is a $\theta$-tallness embedding with closure $\lambda$. The forward directions in Main Theorem (3) and (4) are similar.

It remains to prove the backward directions of Main Theorem (2) - (4). To do this we will start with an embedding $j:V\to M$ witnessing the $\theta$-tallness with closure $\lambda$ of $\kappa$, force to violate $\GCH$ at either $\kappa$ or $\lambda$, and then lift the embedding to the forcing extension. In order to lift the embedding we will use Woodin's method of surgery to modify a certain generic $g$ to obtain $g^*$ with the pullback property $j"H\subseteq g^*$. The following lemma, due to Woodin, will allow us to show that $g^*$ is a generic filter.

\begin{keylemma}\label{surgery}
Suppose $N$ and $M$ are transitive inner models of $\ZFC$ and $j:N\to M$ is a nontrivial elementary embedding with critical point $\kappa$ that is continuous at regular cardinals $\geq\lambda^+$ where $\lambda\geq\kappa$. Then if $A\in M$ is such that $|A|^M\leq j(\lambda)$ then $|A\cap \ran(j)|^V\leq\lambda$.
\end{keylemma}

\begin{proof}
Let $j:N\to M$ and $A\in M$ be as above; that is, $|A|\leq j(\lambda)$. 

First we will argue that it suffices to consider the case where $A$ is a set of ordinals. Let $\vec{B}:=\langle b_\alpha\mid\alpha<\beta\rangle\in N$ be a sequence of length $\beta$ such that $A\subseteq \ran(j(\vec{B}))$; for example, $\vec{B}$ could be an enumeration of some sufficiently large $V_\theta^N$ so that $j(\vec{B})$ is an enumeration of $V_{j(\theta)}^M$. Clearly $j(\vec{B})$ is a sequence of length $j(\beta)$ in $M$, write $j(\vec{B})=\langle b'_\alpha\mid\alpha<j(\beta)\rangle$. Let $A_0=\{\alpha<j(\beta)\mid b'_\alpha\in A\}$. Then $A_0\in M$ and we have $|A_0|^M=|A|^M$. Clearly $b'_\alpha\in\ran(j)$ if and only if for some $\xi<\beta$ we have $b'_\alpha=j(b_\xi)=b'_{j(\xi)}$. In other words, $b'_\alpha\in \ran(j)$ if and only if $\alpha\in\ran(j)$. It follows that $|A\cap\ran(j)|^V=|A_0\cap\ran(j)|^V$, hence we have reduced to the case in which $A$ is a set of ordinals.

Suppose $A\in M$ is a set of ordinals with $|A|^M\leq j(\lambda)$ and $|A\cap\ran(j)|^V\geq\lambda^+$. Then $A$ contains $\lambda^+$-many elements of the form $j(\alpha)$ for $\alpha\in N$. That is, we may assume $A$ contains elements of the form $j(\beta_\alpha)$ where $\alpha<\lambda^+$ and $\langle\beta_\alpha\mid\alpha<\lambda^+\rangle\in V$ is a strictly increasing sequence of ordinals which is not necessarily in $N$ since it was defined using $A\in M$. Now let $\delta=\sup\langle\beta_\alpha\mid\alpha<\lambda^+\rangle$. Furthermore, we know that $\cf(\delta)^V=\lambda^+$ and hence $\cf(\delta)^N\geq\lambda^+$. This implies that $\cf(j(\delta))^M\geq j(\lambda^+)$. Since $j$ is continuous at regular cardinals $\geq\lambda^+$, and thus at $\cf(\delta)^N$, we know that $A$ contains unboundedly many $j(\beta_\alpha)$ less than $j(\delta)$. So in $M$, $A$ is unbounded in $j(\delta)$, but this implies that $|A|^M\geq j(\lambda^+)$ which contradicts our assumption that $|A|^M\leq j(\lambda)$.

\end{proof}


The following theorem suffices to finish the proof of Main Theorem (2) - (4).

\begin{theorem}\label{alllambda}
For any cardinals $\kappa\leq\lambda\leq\theta$, if $\kappa$ is $\lambda$-supercompact and $\theta$-tall then there is a forcing extension in which $\kappa$ is $\lambda$-supercompact and $2^\kappa\geq\theta$; and hence also $2^\lambda\geq\theta$. Indeed, the forcing preserves cardinals on $[\kappa,\lambda^+]\cup (2^\lambda,\infty)$ and assuming $\GCH$ holds at $\lambda$, all cardinals $\geq\kappa$ are preserved.
\end{theorem}


In the following proof of Theorem \ref{alllambda} we will use Woodin's method of surgery referred to just before the above key lemma.

\begin{proof}[Proof of Theorem \ref{alllambda}]\

\par

\subsection{Setup}

Let $\kappa$ be $\lambda$-supercompact and $\theta$-tall. As before, by the remarks at the end of subsection \ref{tallness}, we may assume without loss of generality that $\lambda^{<\kappa}=\lambda$ and $\theta^\lambda=\theta$. We may further assume that $2^\lambda=\lambda^+$ since the forcing to achieve this is $\leq\lambda$-distributive and thus preserves the $\lambda$-supercompactness and $\theta$-tallness of $\kappa$. By Lemma \ref{generators} we have an elementary embedding $j:V\to M$ with $\cp(j)=\kappa$, $j(\kappa)>\theta$, $M^\lambda\subseteq M$, and 
$$M=\{j(h)(j"\lambda,\alpha)\mid\alpha\leq\theta \textrm{ and } h:P_\kappa\lambda\times\kappa\to V\textrm{ is a function in } V\}.$$ 
By Lemma \ref{menaslemma} we may assume without loss of generality that there is a partial function $f\subseteq\kappa\times\kappa$ in $V$ such that $j(f)(\kappa)>\theta$. Let $\P$ be the lottery preparation relative to $f$. Let $G\subseteq\P$ be $V$-generic and let $\Q=\Add(\kappa,\theta)^{V[G]}$. Let $H\subseteq \Q$ be $V[G]$-generic. Notice that since $\P$ has size $\kappa$ and $\Q$ is $\kappa^+$-c.c. it follows that $\P*\Q$ is $\kappa^+$-c.c., and thus by Lemma \ref{chainlemma} that $M[G][H]$ is closed under $\lambda$-sequences in $V[G][H]$.

\subsection{Lifting $j$ through the lottery preparation}

By elementarity $j(\P)$ is the lottery preparation of length $j(\kappa)$ relative to $j(f)$ as defined in $M$. Since $M$ is closed under $\lambda$-sequences in $V$ it follows that the iterations $\P$ and $j(\P)$ agree up to stage $\kappa$ and since $\Q\in M[G]$ is $<\kappa$-closed it appears in the stage $\kappa$ lottery in $j(\P)$. Hence we may factor $j(\P)$ below a condition $p\in j(\P)$ that opts for $\Q$ at stage $\kappa$ as $j(\P)\restrict p\cong \P*\Q*\P_{tail}$. Since $j(f)(\kappa)>\theta$ it follows that the next nontrivial stage of forcing in $j(\P)$ is beyond $\theta$ and hence that $\P_{tail}$ is a term for $\leq\theta$-closed forcing. As in the proof of Theorem \ref{indestructibility} we will construct a descending $\lambda^+$-sequence of conditions in $V[G][H]$ that meets every dense subset of $\P_{tail}$ in $M[G][H]$. Let $D$ be a dense subset of $\P_{tail}$ in $M[G][H]$ and let $\dot{D}\in M$ be a nice $\P*\Q$-name for $D$. Then $D=j(h_{\dot{D}})(j"\lambda,\alpha)_{G*H}$ for some $\alpha\leq\theta$ and some function $h_{\dot{D}}$ with domain $P_\kappa\lambda\times\kappa$ and range contained in the set of nice names for dense subsets of a tail of $\P$. Since the sequence of names $\langle j(h_{\dot{D}})(j"\lambda,\alpha)\mid\alpha\leq\theta\rangle$ is in $M$ and $\P_{tail}$ is $\leq\theta$-closed in $M[G][H]$ we can find a condition in $\P_{tail}$ meeting every dense set mentioned by the sequence $\langle j(h_{\dot{D}})(j"\lambda,\alpha)\mid\alpha\leq\theta\rangle$. Since there are $\leq\lambda^+$-many functions from $P_\kappa\lambda\times\kappa$ to the set of nice names for dense subsets of a tail of $\P$, it follows from the fact that $M[G][H]$ is closed under $\lambda$-sequences in $V[G][H]$ that we can construct a descending $\lambda^+$-sequence in $V[G][H]$ that meets each dense subset of $\P_{tail}$ in $M[G][H]$. Let $G_{tail}$ be the $M[G][H]$-generic filter generated by this descending sequence. Then we may lift the embedding in $V[G][H]$ to $j:V[G]\to M[j(G)]$ where $j(G)=G*H*G_{tail}$ and since $G_{tail}\in V[G][H]$ we may use Lemma \ref{groundlemma} to see that  $M[j(G)]$ is closed under $\lambda$-sequences in $V[G][H]$.


\subsection{The factor diagram}

Let $X=\{j(h)(j"\lambda,\theta)\mid h:P_\kappa\lambda\times\kappa\to V[G] \textrm{ where $h\in V[G]$}\}$. Then $X\elesub M[j(G)]$ and we let $k:M_0'\to M[j(G)]$ be the inverse of the Mostowski collapse $\pi:X\to M_0'$ and let $j_0:V[G]\to M_0'$ be defined by $j_0:=k^{-1}\circ j$. It follows that $j_0$ is the ultrapower embedding by the measure $\mu:=\{X\subseteq P_\kappa\lambda\times\kappa\mid (j"\lambda,\theta)\in j(X)\}$ where $\mu\in V[G][H]$. By elementarity, $M_0'$ is of the form $M_0[j_0(G)]$, where $M_0\subseteq M_0'$ and $j_0(G)\subseteq j_0(\P)\in M_0'$ is $M_0$-generic. Furthermore, $j_0(G)=G*H_0*G^0_{tail}$ where $H_0$ is $M_0[G]$-generic for $\Add(\kappa,\pi(\theta))^{M_0[G]}$ and $G^0_{tail}$ is $M[G][H_0]$-generic for the tail of the iteration $j_0(\P)$ above $\kappa$. The following diagram is commutative.
$$
\xymatrix{
V[G] \ar[r]^j \ar[rd]_{j_0} & M[j(G)] \\
						 & M_0[j_0(G)] \ar[u]_k \\
}
$$
It follows that $j_0$ is a class of $M_0[j_0(G)]$ which is closed under $\lambda$-sequences in $V[G][H_0]$ and that $j_0(\kappa)>\pi(\theta)$.


\subsection{Outline of the rest of the proof}\label{outline}

We would like to lift $j$ through the stage $\kappa$ forcing, $\Q$. This cannot be accomplished using a master condition argument since $|j"H|=\theta$. In order to lift the embedding we will force with $j_0(\Q)$ over $V[G][H]$ to obtain a generic $g_0$ for $j_0(\Q)$. In subsection \ref{obtainingsection} we will argue that $k"g_0$ generates an $M[j(G)]$-generic $g$ for $j(\Q)$. However, we have no reason to expect that $j"H\subseteq g$ and thus we need to do more work in order to lift the embedding. In subsection \ref{surgerysection} we will use Woodin's method of surgery to modify the filter $g$ to obtain an $M[j(G)]$-generic $g^*$ for $j(\Q)$ with $j"H\subseteq g^*$. Then the embedding lifts to 
\begin{align}
j:V[G][H]\to M[j(G)][g^*]\label{surgerylift}
\end{align}
in $V[G][H][g_0]$ where $j(H)=g^*$.

The embedding (\ref{surgerylift}) does not witness that $\kappa$ is $\lambda$-supercompact in $V[G][H]$ because the embedding is a class of $V[G][H][g_0]$. Under the assumption that we have lifted the embedding as in (\ref{surgerylift}) we will now show that the embedding lifts further to our target model $V[G][H][g_0]$ witnessing that $\kappa$ is $\theta$-tall and $\lambda$-supercompact in $V[G][H][g_0]$. Furthermore, we will show in subsection \ref{preservescardinals} that 
\begin{align}
(2^\kappa\geq\theta)^{V[G][H][g_0]}.\label{GCHresultsalllambda}
\end{align}

Let us first argue, assuming we have $g^*$ as above, that $M[j(G)][g^*]$ is closed under $\lambda$-sequences in $V[G][H][g_0]$. We now show that $j_0(\Q)$ is $\leq\lambda$-distributive in $V[G][H]$. Since $j_0(\kappa)>\lambda$ it follows that $j_0(\Q)$ is $\leq\lambda$-closed in $M_0[j_0(G)]$. Since $M_0[j_0(G)]$ is closed under $\lambda$-sequences in $V[G][H_0]$ it follows that $j_0(\Q)$ is $\leq\lambda$-closed in $V[G][H_0]$ and since $\leq\lambda$-closed forcing remains $\leq\lambda$-distributive in $\lambda^+$-c.c. forcing extensions, it follows that $j_0(\Q)$ is $\leq\lambda$-distributive in $V[G][H]$. Since $M[j(G)]$ is closed under $\lambda$-sequences in $V[G][H]$ and $j_0(\Q)$ is $\leq\lambda$-distributive in $V[G][H]$ it easily follows that $M[j(G)]$ is closed under $\lambda$-sequences in $V[G][H][g_0]$. Since $g^*$ is constructed from $g_0$ we have $g^*\in V[G][H][g_0]$ and it follows that $M[j(G)][g^*]$ is closed under $\lambda$-sequences of ordinals in $V[G][H][g_0]$. By using a well ordering of a sufficient initial segment of the universe $M[j(G)][g^*]$, it follows that $M[j(G)][g^*]$ is closed under $\lambda$-sequences in $V[G][H][g_0]$.

Now we show that the embedding (\ref{surgerylift}) lifts through $j_0(\Q)$. Every element of $M[j(G)][g^*]$ is of the form $j(h)(j"\lambda,\alpha)$ where $h:P_\kappa\lambda\times\kappa\to V[G][H]$ is in $V[G][H]$ and $\alpha\leq\theta$. In other words, $M[j(G)][g^*]$ is generated by $\{(j"\lambda,\alpha)\mid\alpha\leq\theta\}\subseteq P_\kappa\lambda\times\kappa$ over $V[G][H]$. Since $P_\kappa\lambda\times\kappa$ has size $\lambda$ and $j_0(\Q)$ is $\leq\lambda$-distributive, it follows by Lemma \ref{lambdadist} that $j"g_0$ generates an $M[j(G)][g^*]$-generic filter $j(g_0)$ for the poset $j(j_0(\Q))$. Thus $j$ lifts in $V[G][H][g_0]$ to $$j:V[G][H][g_0]\to M[j(G)][g^*][j(g_0)].$$
Since $j(g_0)\in V[G][H][g_0]$ and $M[j(G)][g^*]$ is closed under $\lambda$-sequences in $V[G][H][g_0]$ it follows that $M[j(G)][g^*][j(g_0)]$ is closed under $\lambda$-sequences in $V[G][H][g_0]$. Since $j$ is a lift of the original embedding it still satisfies $j(\kappa)>\theta$. Hence $j$ witnesses that $\kappa$ is $\theta$-tall with closure $\lambda$ in $V[G][H][g_0]$.

To complete the proof of Theorem \ref{alllambda} it remains to carry out the surgery argument and to argue that (\ref{GCHresultsalllambda}) holds.

\subsection{Obtaining the generic on which to perform surgery}\label{obtainingsection}

Let $g_0$ be as in subsection \ref{outline}; that is, $g_0$ is $V[G][H]$-generic for $j_0(\Q)$. In this subsection we will argue that $k"g_0$ generates an $M[j(G)]$-generic for $j(\Q)$. Each $x\in M[j(G)]$ is of the form $x=j(h)(j"\lambda,\alpha)$ for some $\alpha\leq\theta$ and some $h:P_\kappa\lambda\times\kappa\to V[G]$ with $h\in V[G]$. Since $j_0"\lambda\in M_0$ it follows that each $x\in M[j(G)]$ is of the form $k(h)(\alpha)$ for some $\alpha\leq\theta$ and some $h:j_0(\kappa)\to M_0[j_0(G)]$ with $h\in M_0[j_0(G)]$; in fact since $k(h\restrict\pi(\theta))$, where $\pi(\theta)$ is the collapse of $\theta$, still has every $\alpha\leq\theta$ in its domain, we may assume that $h:\pi(\theta)\to M_0$. Let $D$ be an open dense subset of $j(\Q)$ in $M[j(G)]$. Then $D=k(\vec{D})(\alpha)$ for some fixed $\alpha\leq\theta$ where $\vec{D}=\langle D_\beta \mid \beta<\pi(\theta)\rangle$ is a sequence of dense open subsets of $j_0(\Q)$. Since $j_0(\Q)$ is $\leq\pi(\theta)$-distributive we know that $\bar{D}:=\bigcap_{\beta<\pi(\theta)}(D_\beta)$ is open dense in $j_0(\Q)$. Hence there is a condition $p\in g_0\cap \bar{D}$. Then $k(p)\in k"g_0\cap k(\bar{D})$. Now $\bar{D}\subseteq D_\beta=\vec{D}(\beta)$ for each $\beta<\pi(\theta)$ and this implies $k(\bar{D})\subseteq k(\vec{D})(\beta)$ for each $\beta<k(\pi(\theta))=\theta$. It follows that 
$$ k(\bar{D})\subseteq k(\vec{D})(\alpha)=D$$
and hence $k(p)\in k"g_0\cap k(\bar{D})\subseteq D$. Therefore $k"g_0$ generates an $M[j(G)]$-generic filter for $j(\Q)$.





\subsection{Surgery}\label{surgerysection}

Now that we have an $M[j(G)]$-generic $g$ for $j(\Q)$ we use Woodin's method of surgery to obtain an $M[j(G)]$-generic $g^*$ for $j(\Q)$ with $j"H\subseteq g^*$. We define $g^*$ in terms of $g$ and $j"H$ in the following way. Let $\Delta$ be the set of coordinates $(\alpha,\beta)\in j(\theta)\times j(\kappa)$ such that there is a $p\in H$ such that $(\alpha,\beta)\in\dom(p)$ and $j(p)(\alpha,\beta)\neq g(\alpha,\beta)$ and let $\pi: j(\Q)\to j(\Q)$ be the automorphism induced by flipping bits over coordinates in $\Delta$. Then we let $g^*:=\pi"g$. In other words, we obtain the modified generic $g^*$ by using $g$ except that whenever $g$ and $j"H$ disagree, we change $g$ to match $j"H$.

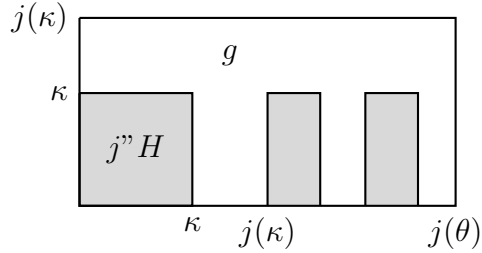
\begin{center}
\begin{tikzpicture}


\filldraw[thick,fill=gray!30] (0, 1.5) node [left] {$\kappa$} --  (1.5,1.5) -- (1.5,0) node [below] {$\kappa$} -- (0,0) -- (0,1.5);
\filldraw[thick,fill=gray!30] (2.5,0) node [below] {$j(\kappa)$} -- (2.5,1.5) -- (3.2,1.5) -- (3.2,0);
\filldraw[thick,fill=gray!30] (3.8,0) -- (3.8,1.5) -- (4.5,1.5) -- (4.5,0);


\node at ( .75,.75) [circle,draw=none] {$j"H$};
\node at ( 2,2) [circle,draw=none] {$g$};

\draw[thick] (0, 2.5) -- (5, 2.5) -- (5,0) node [below] {$j(\theta)$} -- (0,0) -- (0,2.5)
	node [left] {$j(\kappa)$};


\end{tikzpicture}
\end{center}


Since $j$ is continuous at regular cardinals $\geq\lambda^+$ the key lemma applies and we use this to show that $g^*$ is a generic filter on $j(\Q)$. First note that if $p\in j(\Q)$ then $|p|^{M[j(G)]}<j(\kappa)$ and so the set of coordinates on which $p^*:=\pi(p)\neq p$ has size $\leq\lambda$ by the key lemma and is thus in $M[j(G)]$ since $M[j(G)]^\lambda\subseteq M[j(G)]$ in $V[G][H][g_0]$. This implies that $p^*\in M[j(G)]$ and thus that $g^*$ defines a filter in $M[j(G)]$.

Now we show that $g^*$ is $M[j(G)]$-generic for $j(\Q)$. Let $A\subseteq j(\Q)$ be a maximal antichain in $M[j(G)]$. Since $j(\Q)$ has the $j(\kappa)^+$-c.c. we have $|A|^{M[j(G)]}\leq j(\kappa)$. Furthermore, each $p\in A$ has $|p|^{M[j(G)]}<j(\kappa)$. Hence $|\bigcup_{p\in A}\dom(p)|^{M[j(G)]}\leq j(\kappa)$. By the key lemma, the set of coordinates mentioned by conditions in $A$ that were involved in the changes we made from $g$ to $g^*$ has size $\leq \lambda$, call this set $\Delta_A$. In other words, $\Delta_A:=\Delta\cap\left(\bigcup_{p\in A}\dom(p)\right)$. Let $\pi_A:j(\Q)\to j(\Q)$ be the automorphism induced by flipping bits over coordinates in $\Delta_A$. The coordinates of bits that get flipped by $\pi_A$ are contained in the domain of the antichain (see the shaded region in the figure below).

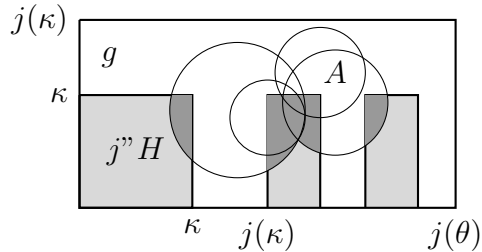
\begin{center}
\begin{tikzpicture}


\def\recta{(0, 1.5) node [left] {$\kappa$} --  (1.5,1.5) -- (1.5,0) node [below] {$\kappa$} -- (0,0) -- (0,1.5)}
	
\def\rectb{(2.5,0) node [below] {$j(\kappa)$} -- (2.5,1.5) -- (3.2,1.5) -- (3.2,0)}

\def\rectc{(3.8,0) -- (3.8,1.5) -- (4.5,1.5) -- (4.5,0)}

\filldraw[thick,fill=gray!30] \recta;

\filldraw[thick,fill=gray!30] \rectb;

\filldraw[thick,fill=gray!30] \rectc;


\node at ( .75,.75) [circle,draw=none] {$j"H$};
\node at ( 0.4,2) [circle,draw=none] {$g$};

\draw[thick] (0, 2.5) -- (5, 2.5) -- (5,0) node [below] {$j(\theta)$} -- (0,0) -- (0,2.5)
	node [left] {$j(\kappa)$};


\def\circlea{(3.4,1.4) circle (0.7)}
\def\circleb{(2.6,1.5) circle (0.5)}
\def\circlec{(2.5,1.2) circle (0.5)}
\def\circled{(2.1,1.3) circle (0.9)}
\def\circlee{(3.2,1.8) circle (0.6)}

\begin{scope}
\clip \circled;
\fill[gray!80] \recta;
\end{scope}

\begin{scope}
\clip \circled;
\fill[gray!80] \rectb;
\end{scope}

\begin{scope}
\clip \circlea;
\fill[gray!80] \rectb;
\end{scope}

\begin{scope}
\clip \circlea;
\fill[gray!80] \rectc;
\end{scope}

\draw \circlea;
\draw \circlec;
\draw \circled;
\draw \circlee;

\node at ( 3.4,1.8) [circle,draw=none] {$A$};

\end{tikzpicture}
\end{center}
Since $|\Delta_A|\leq\lambda$ we have $\Delta_A\in M[j(G)]$ and it follows that $\pi_A\in M[j(G)]$. Then $\pi_A^{-1}"A$ is a maximal antichain of $j(\Q)$ and by genericity of $g$ there is a condition $p\in  g$ that decides $\pi_A^{-1}"A$. It follows that $\pi(p)\in g^*$  decides $A$ since $\pi"A = \pi_A"A$. This establishes that $g^*$ is $M[j(G)]$-generic for $j(\Q)$.



Since we arranged $j"H\subseteq g^*$ by definition, we may use Lemma \ref{liftingcriterion} to lift the embedding to $j:V[G][H]\to M[j(G)][j(H)]$ where $j(H)=g^*$. Since we used $g_0$ to define $g^*$, this lift is a class of $V[G][H][g_0]$. We argued above (in the outline given in section \ref{outline}) that we can lift the embedding further through the $g_0$ forcing. So all that remains is to show that $j_0(\Q)$ preserves cardinals and that $2^\kappa\geq\theta$ in $V[G][H][g_0]$.

\subsection{Preserving $2^\kappa\geq\theta$ in $V[G][H][g_0]$}\label{preservescardinals}

We have already argued that $j_0(\Q)$ is $\leq\lambda$-distributive in $V[G][H]$ and we will now argue that $j_0(\Q)$ is $\lambda^{++}$-c.c. From this it follows that $j_0(\Q)$ preserves cardinals over $V[G][H]$ and $2^\kappa\geq\theta$ in $V[G][H]$. Each condition $p\in j_0(\Q)$ is in $M_0[j_0(G)]$ and is thus of the form $p=j_0(h_p)(j_0"\lambda,\theta)$ for some $h_p:P_\kappa\lambda\times\kappa\to \Q$ with $h\in V[G]$. For each $p\in j_0(\Q)$, $\dom(h_p)$ has size $\lambda$ in $V[G]$ and thus $h_p$ leads to a function $\bar{h}_p:\lambda\to \Q$, which can be viewed as a condition in the full support product of $\lambda$-many copies of $\Q$ as defined in $V[G]$, which we denote by $\overline{\Q}$. We will show that $j_0(\Q)$ is $\lambda^{++}$-c.c. in $V[G][H]$ by arguing that $\bar{\Q}$ is $\lambda^{++}$-c.c. in $V[G][H]$ and that an antichain of $j_0(\Q)$ of size $\lambda^{++}$ in $V[G][H]$ would lead to an antichain of $\bar{\Q}$ of size $\lambda^{++}$ in $V[G][H]$.

\begin{claim}\label{Qbarlemma}
$\bar{\Q}$ is $\lambda^{++}$-c.c. in $V[G][H]$
\end{claim}
\begin{proof}[Proof of claim]

By a delta system argument $\bar{\Q}$ is $\lambda^{++}$-c.c. in $V[G]$. Suppose $A\in V[G][H]$ is an antichain of $\bar{\Q}$ with $|A|=\delta$. We will show that $A$ leads to an antichain of size $\delta$ of $\bar{\Q}\cong \Q\times\bar{\Q}$ in $V[G]$ and thus that $\delta<\lambda^{++}$. Let 
$$q\forces \dot{A}\textrm{ is an antichain of $\bar{\Q}$ and }\dot{f}:\delta\to \dot{A}\textrm{ is bijective}$$ where $q\in \Q\cap H$ and $\dot{A}_H=A$. For each $\alpha<\delta$ let $q_\alpha\leq q$ be such that $q_\alpha\forces \dot{f}(\check{\alpha})=\check{p}_\alpha$ where $p_\alpha\in \bar{\Q}$. We have $\bar{\Q}\cong\Q\times\bar{\Q}$ in $V[G]$ and we now show that $W:=\{(q_\alpha,p_\alpha)\in \Q\times\bar{\Q}\mid\alpha<\delta\}$ is an antichain of size $\delta$ of $\Q\times\bar{\Q}$ in $V[G]$. Clearly $W\in V[G]$ because in choosing the pairs $(q_\alpha,p_\alpha)$ in $W$ we only used the forcing relation $\forces_{\Q}$. Suppose for a contradiction that $W$ is not an antichain, i.e. that $(q^*,p^*)\leq(q_\alpha,p_\alpha),(q_\beta,p_\beta)$ for some $\alpha,\beta<\delta$ with $\alpha\neq\beta$ and some $(q^*,p^*)\in \Q\times\bar{\Q}$. Let $H^*$ be $V[G]$-generic for $\Q$ with $q^*\in H^*$. Since $q^*\leq q$ it follows that $\dot{f}_{H^*}$ enumerates an antichain. Furthermore we have $\dot{f}_{H^*}(\alpha)=p_\alpha$, $\dot{f}_{H^*}(\beta)=p_\beta$, and $p^*\leq p_\alpha, p_\beta$, a contradiction. Hence we conclude that $W$ is an antichain of $\Q\times\bar{\Q}$ in $V[G]$ and since $\bar{\Q}\cong\Q\times\bar{\Q}$ we see that $W$ leads to an antichain of $\bar{\Q}$ of size $\delta$ in $V[G]$. Therefore, $\delta<\lambda^{++}$. Hence we conclude that an antichain $A$ of $\bar{\Q}$ in $V[G][H]$ must have size $<\lambda^{++}$.
\end{proof}

Now we complete the proof of Theorem \ref{alllambda} by showing that $j_0(\Q)$ is $\lambda^{++}$-c.c. in $V[G][H]$. Suppose that in $V[G][H]$, $j_0(\Q)$ has an antichain $A$ of size $\delta$. For each $p\in A\subseteq j_0(\Q)$ let $h_p:P_\kappa\lambda\times\kappa\to \Q$ be such that $p=j_0(h_p)(j"\lambda,\theta)$ where $h_p\in V[G]$. As above each $h_p$ yields a condition in $\bar{\Q}$, call it $\bar{h}_p$. For $p,q\in A$ we have $j_0(h_p)(j"\lambda,\theta)\perp_{j_0(\Q)} j_0(h_q)(j"\lambda,\theta)$, and thus by elementarity we conclude that there is a $(\sigma,\alpha)\in P_\kappa\lambda\times\kappa$ such that $h_p(\sigma,\alpha)\perp h_q(\sigma,\alpha)$. This implies that $\bar{A}:=\{\bar{h}_p\mid p\in A\}$ is an antichain in $\bar{\Q}$ where $|\bar{A}|=\delta$. By Claim \ref{Qbarlemma}, $\delta<\lambda^{++}$. Thus $j_0(\Q)$ is $\lambda^{++}$-c.c. in $V[G][H]$.

Thus we have shown that in $V[G][H][g_0]$, $\kappa$ is $\lambda$-supercompact and $\theta$-tall, and $2^\lambda\geq\theta$.

Let us argue that cardinals in $[\kappa,\lambda^+]\cup(2^\lambda,\infty)$ are preserved. We started with a model and forced $2^\lambda=\lambda^+$ which may have collapsed cardinals in $(\lambda^+, 2^\lambda]$. We then add a fast function using $\kappa^+$-c.c. forcing which preserves cardinals $\geq\kappa$. The remaining forcing is $\P*\Q*j_0(\Q)$ where $\P*\Q$ is $\kappa^+$-c.c. and $j_0(\Q)$ preserves cardinals over $V[G][H]$. Thus in the final model $V[G][H][g_0]$ we have preserved cardinals in the interval $[\kappa,\lambda^+]\cup (2^\lambda,\infty)$.

\end{proof}


\end{document}